\documentclass[10pt,draft]{amsart}

\usepackage[margin=2.5cm,bottom=2.5cm,top=2.5cm]{geometry}
\usepackage{mathtools}
\mathtoolsset{showonlyrefs}
\usepackage{amsmath,amssymb,amsthm}
\usepackage{amsfonts}
\usepackage{latexsym}
\usepackage{amscd}
\usepackage{xcolor}

\begin{document}
\newtheorem{lemme}{Lemma}[section]
\newtheorem{proposition}{Proposition}[section]
\newtheorem{cor}{Corollary}[section]
\numberwithin{equation}{section}
\newtheorem{theoreme}{Theorem}[section]
\theoremstyle{remark}
\newtheorem{example}{Example}[section]
\newtheorem*{ack}{Acknowledgment}
\theoremstyle{definition}
\newtheorem{definition}{Definition}[section]
\theoremstyle{remark}
\newtheorem*{notation}{Notation}
\theoremstyle{remark}
\newtheorem{remark}{Remark}[section]
\newenvironment{Abstract}
{\begin{center}\textbf{\footnotesize{Abstract}}%
\end{center} \begin{quote}\begin{footnotesize}}
{\end{footnotesize}\end{quote}\bigskip}
\newenvironment{nome}
{\begin{center}\textbf{{}}%
\end{center} \begin{quote}\end{quote}\bigskip}

\newcommand{\triple}[1]{{|\!|\!|#1|\!|\!|}}
\newcommand{\xx}{\langle x\rangle}
\newcommand{\ep}{\varepsilon}
\newcommand{\al}{\alpha}
\newcommand{\be}{\beta}
\newcommand{\de}{\partial}
\newcommand{\la}{\lambda}
\newcommand{\La}{\Lambda}
\newcommand{\ga}{\gamma}
\newcommand{\del}{\delta}
\newcommand{\Del}{\Delta}
\newcommand{\sig}{\sigma}
\newcommand{\ome}{\Omega^n}
\newcommand{\Ome}{\Omega^n}
\newcommand{\C}{{\mathbb C}}
\newcommand{\N}{{\mathbb N}}
\newcommand{\Z}{{\mathbb Z}}
\newcommand{\R}{{\mathbb R}}
\newcommand{\T}{{\mathbb T}}
\newcommand{\Rn}{{\mathbb R}^{n}}
\newcommand{\Rnu}{{\mathbb R}^{n+1}_{+}}
\newcommand{\Cn}{{\mathbb C}^{n}}
\newcommand{\spt}{\,\mathrm{supp}\,}
\newcommand{\Lin}{\mathcal{L}}
\newcommand{\SSS}{\mathcal{S}}
\newcommand{\F}{\mathcal{F}}
\newcommand{\xxi}{\langle\xi\rangle}
\newcommand{\eei}{\langle\eta\rangle}
\newcommand{\xei}{\langle\xi-\eta\rangle}
\newcommand{\yy}{\langle y\rangle}
\newcommand{\dint}{\int\!\!\int}
\newcommand{\hatp}{\widehat\psi}
\renewcommand{\Re}{\;\mathrm{Re}\;}
\renewcommand{\Im}{\;\mathrm{Im}\;}
\title[
$H^1$ scattering for mass subcritical short-range NLS
 ]
{$H^1$ scattering for mass-subcritical NLS \\with short-range nonlinearity and initial data in $\Sigma$}
 \author{N. Burq}
\address{D\'epartment de Math\'ematiques Universit\'e Paris-Saclay, Bat. 307, 91405 Orsay Cedex
France}
\email{nicolas.burq@universite-paris-saclay.fr}
 \author{V. Georgiev}
\address{Dipartimento di Matematica, Universit\`a di Pisa, Italy}
\email{georgiev@dm.unipi.it}
 \author{N. Tzvetkov}
\address{D\'epartment de Math\'ematiques  CY  Cergy Paris Universit\`e, 2 av. Adolphe Chauvin, 95302 Cergy-Pontoise Cedex, France}
\email{nikolay.tzvetkov@cyu.fr}
\author{N. Visciglia}
\address{Dipartimento di Matematica, Universit\`a di Pisa, Italy}
\email{nicola.visciglia@unipi.it}
\thanks{N.T.  was supported by ANR grant ODA (ANR-18-CE40-0020-01), 
V.G. and N.V. by PRIN grant 2020XB3EFL, moreover they acknowledge the Gruppo Nazionale per l' Analisi Matematica, la Probabilit\`a e le loro Applicazioni (GNAMPA) of the Istituzione Nazionale di Alta Matematica (INDAM)}
\date{\today}
 \maketitle
 \par \noindent
\begin{abstract}

We consider short-range mass-subcritical
nonlinear Schr\"odinger equations
and we show that the corresponding solutions
with initial data in $\Sigma$ scatter in $H^1$. Hence we up-grade the classical scattering result proved by Yajima and Tsutsumi
from $L^2$ to $H^1$.
We also provide some partial results concerning the scattering of the first order moments, as well as a short proof via lens transform of a classical result
due to Tsutsumi and Cazenave-Weissler on the scattering in $\Sigma$.
\end{abstract}
\section{Introduction}
In this paper we are interested in the long-time behavior of solutions to the following Cauchy problems
associated with the defocusing nonlinear Schr\"odinger equations (NLS):
\begin{equation}\label{u}\begin{cases}
i\partial_t u +\Delta u -u|u|^{p}=0, \quad (t, x)\in \R\times \R^n\\
u(0,.)=\varphi.
\end{cases}
\end{equation}
It is well--known, by combining Strichartz estimates and a contraction argument, that the Cauchy problems  above are locally well-posed for every initial datum $\varphi\in H^1(\R^n)$
with time of existence which depends only from the size of the initial datum in $H^1(\R^n)$,
provided that $0<p<\frac 4{n-2}$ if $n\geq 3 $ and $0<p<\infty$ if $n=1,2$. Then the conservation of mass and conservation of the energy:
\begin{equation}\label{energy}E(u(t,x))=\frac 12 \int_{\R^n} |\nabla u(t,x)|^2 dx + \frac 1{p+2}\int_{\R^n} |u(t,x)|^{p+2} dx,\end{equation}
(the energy is positive since we consider the defocusing NLS) imply that the $H^1(\R^n)$ norm of the solution is uniformly bounded
and hence the local theory can be iterated in order to provide a global well-posedness result.
There is a huge literature around this topic, for simplicity we quote the very complete book \cite{c} and all the references therein.
We also recall that the much more difficult critical nonlinearity $p=\frac 4{n-2}$ for $n\geq 3$
has been extensively studied more recently starting from the pioneering paper \cite{bourgain} in the radial case  and its extension in the non radial setting in \cite{ckstt}.
In the sequel, in order to emphasize the dependence of the nonlinear solution
from the initial datum, we shall write $u_\varphi(t,x)$
to denote the unique global solution to \eqref{u} where $\varphi\in H^1(\R^n)$ and $p$ is assumed to be given.
\\

Once the existence of global solutions is established then it is natural to look at the long time behavior. In the range of mass-supercritical and energy-subcritical nonlinearities, namely
$\frac 4n <p < \frac 4{n-2}$ for $n\geq 3$ and $\frac 4n< p<\infty$ for $n=1,2$, it has been proved
that nonlinear solutions to NLS behave as free waves as $t\rightarrow \pm \infty$. More precisely we have
the following property:
\begin{multline}\label{scatteringH1}
\forall \varphi \in H^1(\R^n) \quad \exists \varphi_\pm \in H^1(\R^n) \hbox{ s.t. }
\|u_\varphi(t,x)-e^{it\Delta} \varphi_\pm\|_{H^1(\R^n)}\overset{t\rightarrow \pm \infty}\longrightarrow 0,\\
\hbox{ provided that } \frac 4n <p < \frac 4{n-2} \hbox{ for } n\geq 3, \quad \frac 4n< p<\infty \hbox{ for } n=1,2.
\end{multline}
We point-out that the scattering property \eqref{scatteringH1} can be stated in the following equivalent form
\begin{equation}\label{scatteringH1equiv}
\|e^{-it\Delta} u_\varphi(t,x)- \varphi_\pm\|_{H^1(\R^n)}\overset{t\rightarrow \pm \infty}\longrightarrow 0
 \end{equation}
by using the fact that the group $e^{it\Delta}$ is an isometry in $H^1(\R^n)$.
The property \eqref{scatteringH1}  is known in the literature as the asymptotic completness of the wave operator in $H^1(\R^n)$,
or more quickly $H^1(\R^n)$ scattering. Roughly speaking \eqref{scatteringH1} implies that
for large times (both positive and negative) the nonlinear evolution can be approximated in $H^1$ by a linear one with a suitably modified initial data which represents the nonlinear effect.
The literature around $H^1(\R^n)$ scattering in the mass-supercritical and energy-subcritical case is huge. Beside the already quoted reference \cite{c} and the bibliography
therein,
we mention al least \cite{gv} in the case $n\geq 3$ and
\cite{n} for $n=1,2$. More recently shorter proof of scattering in the energy space $H^1(\R^n)$ for mass-supercritical and energy-subcritical NLS has been achieved by using the interaction Morawetz estimates, first
introduced in \cite{ckstt}. We mention in this direction \cite{cgt}, \cite{pv}, \cite{vis} and all the references therein.
We recall that the scattering of nonlinear solutions to free waves in the energy space has been extended to  the energy critical case, namely
$p=\frac 4{n-2}$ when $n\geq 3$, in a series of papers starting from the pioneering articles
\cite{bourgain} and \cite{ckstt} for $n=3$. Its extension in higher dimension is provided in \cite{v}. In the mass critical case $p=\frac 4n$ the $H^1(\R^n)$ scattering property follows from
\cite{do}, \cite{do1}, \cite{do2}.\\

Notice that  the mass subcritical nonlinearities, namely $0<p<\frac 4n$, do not enter in the analysis above. In fact we can introduce the intermediate nonlinearity
$p=\frac 2n$ which is a discriminant between short-range ($\frac 2n<p<\frac 4n$) and long-range nonlinearity
($0<p\leq \frac 2n$). More specifically one can prove that in the long-range mass-subcritical setting
nonlinear solutions do not behave as free waves.
In this direction we mention \cite{barab} and \cite{c}, where it is proved that
the scattering property fails in the $L^2(\R^n)$ topology for every nontrivial solution to NLS, even for initial datum very smooth.
The precise statement can be given in the following form:
\begin{multline}\label{barab}
\limsup_{t\rightarrow \pm \infty}
\|u_\varphi(t,x)-e^{it\Delta} \psi\|_{L^2(\R^n)}>0, \\
\quad \forall (\varphi, \psi)\in C^\infty_0(\R^n) \times L^2(\R^n), \quad (\varphi, \psi)\neq (0,0), \hbox{ provided that } 0<p\leq \frac 2n.
\end{multline}

On the contrary in the short-range mass-subcritical case, following \cite{yt}, one can show the following version of scattering:
\begin{equation}\label{scatteringsigma}
\forall \varphi \in \Sigma_n \quad \exists \varphi_\pm \in L^2(\R^n) \hbox{ s.t. } \\
\|u_\varphi(t,x)-e^{it\Delta} \varphi_\pm\|_{L^2(\R^n)}\overset{t\rightarrow \pm \infty}\longrightarrow 0, \hbox{ provided that } \frac 2n<p<\frac 4n
\end{equation}
where the space $\Sigma_n$ is the following one
$$\Sigma_n=\Big \{\varphi \in H^1(\R^n) | \int_{\R^n} |x|^2 |\varphi|^2 dx<\infty \Big \}$$
endowed with the following norm
$$\|\varphi\|_{\Sigma_n}^2=\int_{\R^n} (|\nabla \varphi|^2 + |\varphi|^2 + |x|^2 |\varphi|^2) dx.$$
Notice that the result in \cite{yt} is very general, in the sense that the full set of short-range mass-subcritical nonlinearities $\frac 2n<p<\frac 4n$ is covered,
and is sharp in view of the aforementioned result in \cite{barab}.
However the weakness of \eqref{scatteringsigma} is that although the initial datum is assumed to belong to the space $\Sigma_n$, the convergence to free waves
is proved only in the $L^2(\R^n)$ sense.\\

The main aim of this paper is to overcome, at least partially, this fact and to up-grade the convergence
from $L^2(\R^n)$ to $H^1(\R^n)$ by assuming that the initial datum belongs to the space $\Sigma_n$. We can now state the main result of this paper
\begin{theoreme}\label{H1}
Assume $\frac 2n<p<\frac 4n$ then for every $\varphi \in \Sigma_n$ there exist $\varphi_\pm\in H^1(\R^n)$
such that
\begin{equation}\label{scattsigma}
\|u_\varphi(t,x)-e^{it\Delta} \varphi_\pm\|_{H^1(\R^n)}\overset{t \rightarrow \pm \infty} \longrightarrow  0.\end{equation}
\end{theoreme}
We point out that Theorem \ref{H1} is new in the sense that it covers the full set of short-range mass-subcritical nonlinearities
$\frac 2n<p<\frac 4n$, despite to previous results where only a subset of short-range nonlinearities was treated.
We quote in this direction \cite{cw}  and \cite{ts}, where
the following property (which is stronger than $H^1(\R^n)$ scattering) is proved:
\begin{equation}\label{cw}
\forall \varphi \in \Sigma_n \quad \exists \varphi_\pm \in \Sigma_n \hbox{ s.t. } \\
\|e^{-it\Delta} u_\varphi(t,x)- \varphi_\pm\|_{\Sigma_n}\overset{t\pm \infty}\longrightarrow 0, \hbox{ provided that } p_n\leq p<\frac 4n
\end{equation}
where
\begin{equation}\label{pnlens}p_n=\frac{2-n +\sqrt{n^2+12 n+4}}{2n},\end{equation}
i.e. $p_n$ is the larger root of the polynomial $nx^2+(n-2)x-4=0$ (see the appendix for a short proof of \eqref{cw} via the lens transform). One can check
that $p_n>\frac 2n$ for every $n\geq 1$ and hence the results in \cite{cw} and \cite{ts} do not cover the full set of short-range mass-subcritical nonlinearities.
Notice also that, despite the fact that \eqref{scatteringH1} and \eqref{scatteringH1equiv} are equivalent, it is not clear whether or not \eqref{cw}  implies
\begin{equation}\label{cwequ}\|u_\varphi(t,x)- e^{it\Delta}  \varphi_\pm\|_{\Sigma_n}\overset{t\rightarrow \pm \infty}\longrightarrow 0.\end{equation}
In fact it is well-known that, due to the dispersion, the $\Sigma_n$ norm grows quadratically in time along free waves and hence
the group $e^{it\Delta}$ is not uniformly bounded w.r.t. the $\Sigma_n$ topology.
Only in some very few special cases it is proved that  \eqref{cw} implies \eqref{cwequ} (see \cite{be}).
Summarizing the main point in Theorem \ref{H1} is that we cover the full range of nonlinearities
$\frac 2n< p < \frac 4n$, however our conclusion is weaker than \eqref{cw} which on the other hand is available for a more restricted set of nonlinearities.\\

We point out that our approach to prove Theorem \ref{H1} is based only on Hilbert spaces considerations and we don't rely on Strichartz estimates.
In fact Strichartz estimates in collaboration with boundedness of a family of space-time Lebesgue norms
that arise from the pseudoconformal energy, are the key tools in \cite{cw} and \cite{ts}. However in order to close the estimates, following this approach,  some restrictions
appear on the nonlinearity and hence the lower bound $p\geq p_n$ is needed.\\

We also underline that in order to prove Theorem \ref{H1} we take
the result in \cite{yt} (see \eqref{scatteringsigma}) as a black-box and we prove how to go from $L^2(\R^n)$ to $H^1(\R^n)$ convergence.
The proof of Theorem \ref{H1} is obtained as a combination of \cite{yt} and the following result.
\begin{theoreme}\label{main}
Let $\varphi\in \Sigma_n$, $0<p< \frac 4n $ and assume that there exist $\varphi_\pm\in L^2(\R^n)$ such that
\begin{equation}\label{assL2}\|u_\varphi(t,x)-e^{it\Delta} \varphi_\pm\|_{L^2(\R^n)}\overset{t \rightarrow \pm \infty} \longrightarrow  0.\end{equation}
Then we have necessarily $\varphi_\pm \in H^1(\R^n)$ and moreover
\begin{equation}\label{H1convergence}
\|u_\varphi (t,x)-e^{it\Delta} \varphi_\pm \|_{H^1(\R^n)}\overset{t \rightarrow \pm \infty} \longrightarrow  0.\end{equation}
\end{theoreme}
Notice that in Theorem \ref{main} we assume the nonlinearity $p$ to be mass-subcritical (both short-range and long-range), however
it is assumed the abstract assumption \eqref{assL2}, which is forbidden in the long-range case by \eqref{barab}
and is granted in the short-range setting by \eqref{scatteringsigma}.
We also remark that the main point in Theorem \ref{main} is \eqref{H1convergence},
on the contrary the regularity property
$\varphi_\pm \in H^1(\R^n)$ is straightforward and follows by the conservation of mass and energy
which guarantee $\sup_t \|u(t,x)\|_{H^1(\R^n)}<\infty$.
By using this fact, along with an interpolation argument, it is easy to deduce that
under the same assumptions of Theorem \ref{main} one can conclude
\begin{equation}\label{H1convergenceweak}
\|u_\varphi (t,x)-e^{it\Delta} \varphi_\pm \|_{H^{s}(\R^n)}\overset{t \rightarrow \pm \infty} \longrightarrow  0, \quad s\in [0,1).
\end{equation}
However the convergence in $H^1(\R^n)$ stated in Theorem \ref{main} is more delicate to prove and is the main contribution of the paper.\\

Next we do some considerations about the convergence of the second order moments of solutions to \eqref{u} to free waves, if the initial datum belongs to $\Sigma_n$.
We recall that the classical definition of scattering in $\Sigma_n$ (see \cite{c} and all the references therein)
is provided by \eqref{cw}, which unfortunately we are not able to show in the full set of short-range mass-subcritical nonlinearities.
Moreover, as already mentioned above it is unclear how, even if \eqref{cw} is established, one can compare the nonlinear solutions
to free waves as described in \eqref{cwequ}.
On the other hand notice that \eqref{cwequ} is a very strong request since it requires to compare asymptotically quantities which diverge for large times. In fact it is well-known
that for free waves the second order moments grow quadratically and hence the request \eqref{cwequ}
seems to be very hard to prove (in fact it is known in very few cases, see \cite{be}). On the other hand
for free waves with initial datum in $\Sigma_n$ we have
that the renormalized second order moments $\int_{\R^n} \frac{|x|^2}{t^2} |e^{it\Delta \varphi}|^2 dx$ are bounded and we have a precise limit as $t\rightarrow \pm \infty$ (see for instance \cite{tv}).
As a consequence it seems quite natural to understand whether or not
we can compare the renormalized second order moment of the nonlinear solution with the renormalized second order moment
of the free wave. The aim of next result is to show that scattering of renormalized second order moments
is equivalent to the regularity of the scattering state $\varphi_\pm$.
\begin{theoreme}\label{mainmoment}
Let $p, \varphi, \varphi_\pm$ as in Theorem \ref{main} then we have the following equivalence
$$\Big \|\frac{|x|}t (u_\varphi(t,x)- e^{it\Delta} \varphi_\pm)\Big\|_{L^2(\R^n)}\overset{t\rightarrow \pm\infty} \longrightarrow 0 \iff \varphi_\pm\in \Sigma_n.$$
\end{theoreme}
Unfortunately we can prove the property $\varphi_\pm\in \Sigma_n$ only for a subset of short-range mass-subcritical NLS, namely the ones treated
in the references \cite{cw}, \cite{ts}. Indeed once \eqref{cw} is established we get for free $\varphi_\pm\in \Sigma_n$ provided that $\varphi\in \Sigma_n$.
We believe that the property $\varphi_\pm\in \Sigma_n$, which appears in Theorem \ref{mainmoment}, is an interesting question of intermediate difficulty compared with the proof of scattering in $\Sigma_n$
as described in \eqref{cw}. We think it deserves to be investigated in the full set of short-range mass-subcritical nonlinearities.
\\

The last result that we present concerns a further equivalent formulation of the regularity condition $\varphi_\pm\in \Sigma_n$
which has appeared in Theorem \ref{mainmoment}. First of all we introduce
the pesudo-conformal transformation of $u_\varphi(t,x)$ defined as follows:
\begin{equation}\label{defpseudo}
w_\varphi(t,x)=\frac{1}{t^{n/2}} \bar u_\varphi(\frac 1t, \frac xt) e^{i\frac{|x|^2}{4t}}.\end{equation}
The key point in \cite{yt} is the proof of the existence
of the functions $w_\varphi^\pm \in L^2(\R^n)$ such that
\begin{equation}\label{YTL2}
\|w_\varphi(t,x)-w_\varphi^\pm \|_{L^2(\R^n)}\overset{t\rightarrow 0^\pm} \longrightarrow 0.\end{equation}
Recall that we already have by Theorem \ref{main} the property $\varphi_\pm \in H^1(\R^n)$, then
showing $\varphi_\pm\in \Sigma_n$ is equivalent to showing $\int_{\R^n} |x|^2 |\varphi_\pm|^2 dx<\infty$. Next theorem is a further characterization of this property
in terms of the regularity of $w_\varphi^\pm$ that appear in \eqref{YTL2}.
\begin{theoreme}\label{equivreg}
Let $p, \varphi, \varphi_\pm$ as in Theorem \ref{main} and let $w_\varphi^\pm$ be given by \eqref{YTL2},
then we have the following equivalence
$$\int_{\R^n} |x|^2 |\varphi_\pm|^2 dx<\infty \iff w_\varphi^\pm \in \dot H^1(\R^n).$$
\end{theoreme}
Roughly speaking by Theorem \ref{equivreg} the property $\varphi_\pm\in \Sigma_n$ is reduced to study up to the time $t=0$ the $H^1(\R^n)$ regularity of solutions
to \eqref{w} (which is the partial differential equation solved by $w_\varphi(t,x)$) with initial condition at time $t=1$ such that $w_\varphi(1,x)\in \Sigma_n$.
\\

We conclude the introduction by quoting the papers \cite{kmmv} and \cite{kmmv1} where the question of scattering theory is studied in negative Sobolev spaces
for a family of long-range mass-subcritical nonlinearities. In particular a series of conditional scattering results are achieved in the aforementioned papers.
We finally mention \cite{bt} where the authors prove in dimensione $n=1$ new probabilistic results about scattering and smoothing effect of the scattering states in weighted negative Sobolev spaces in the mass-subcritical short-range regime. The result has been extended in higher dimensions under the radiality condition in \cite{l}.
\\

{\bf Acknowledgement.} The authors are grateful to Thierry Cazenave for interesting discussions during the preparation
of this paper.
{\color{red}

}

\section{The pseudo-conformal transformation}

Let $u_\varphi(t,x)$ the unique global solution to \eqref{u} with initial condition $\varphi\in \Sigma_n$, then following \cite{yt} we introduce the pseudo-conformal transformation
$w_\varphi(t,x)$ defined by \eqref{defpseudo}.
Notice that $w_\varphi(t,x)$ is well-defined for $(t,x)\in (0,\infty)\times \R^n$ and $(t,x)\in (-\infty,0)\times \R^n$.
We shall focus mainly on the restriction of $w_\varphi (t,x)$ on the strip $(t,x)\in (0,\infty)\times \R^n$, which is of importance in order to prove Theorem \ref{main}
as $t\rightarrow \infty$, by a similar argument we can treat the case $t\rightarrow -\infty$ by using the restriction of $w_\varphi (t,x)$ on the strip
$(t,x)\in (-\infty,0)\times \R^n$.
One can check by direct computation that $w_\varphi(t,x)$ is solution to the following partial differential equation:
\begin{equation}\label{w}
i\partial_t w_\varphi  +\Delta w_\varphi  - t^{-\alpha(n,p)}w_\varphi |w_\varphi|^{p}=0, \quad (t, x)\in (0,\infty) \times \R^n, \quad \alpha(n,p)=2-\frac {np}2.
\end{equation}
Notice that in the regime of short-range nonlinearity we have that
$t^{-\alpha(n,p)}\in L^1(0,1)$ and in the regime of long-range nonlinearity we have that
$t^{-\alpha(n,p)}\notin L^1(0,1)$. Hence the nonlinearity $p=\frac 2n$ is borderline to guarantee
local integrability in a neighborhood of the origin of the weight $t^{-\alpha(n,p)}$ which appears in front of the nonlinearity in \eqref{w}.
As already mentioned in the introduction, the key idea in \cite{yt} is to deduce the $L^2(\R^n)$ scattering property for the solution $u_\varphi(t,x)$ as $t\rightarrow \infty$
by showing that the following  limit exists
$$\lim_{t\rightarrow 0^+ } w_\varphi(t,x) \hbox{ in } L^2(\R^n).$$ Notice that even if  $w_\varphi(t,x)\in \Sigma_n$ for $t\neq 0$,
it is not well-defined at $t=0$ and hence to show the existence of the limit above
as $t\rightarrow 0^+ $ is not obvious.
In order to achieve this property in \cite{yt} it is first proved that the limit above exists in $L^2(\R^n)$ in the weak sense, and then in a second step
the convergence is up-graded to strong
convergence in $L^2(\R^n)$.\\

We collect in the next proposition the key properties of $w_\varphi(t,x)$ that will be useful in the sequel.
\begin{proposition}
Let $\varphi\in \Sigma_n$, $0<p<\frac 4{n-2}$ for $n\geq 3$ and $0<p<\infty$ for $n=1,2$. Let $w_\varphi(t,x)$ be the pseudoconformal transformation associated with $u_\varphi(t,x)$ as in \eqref{defpseudo}, then we have
the following properties:
\begin{equation}\label{sigmaw}
w_\varphi(t,x)\in C((0,\infty);\Sigma_n)
\end{equation}
and
\begin{equation}\label{YTeq}
t^{\alpha(n,p)} \|\nabla w_\varphi(t,x)\|_{L^2(\R^n)}^2 + \|w_\varphi(t,x)\|_{L^{p+2}(\R^n)}^{p+2}
=\|\nabla w_\varphi(t,x)\|_{L^2}^2 \frac d{dt} t^{\alpha(n,p)}>0.
\end{equation}
In particular for $0<p<\frac 4n$ we have
\begin{equation}\label{YT}
\sup_{t\in (0,1]} \Big (t^{\alpha(n,p)} \|\nabla w_\varphi(t,x)\|_{L^2(\R^n)}^2 + \|w_\varphi(t,x)\|_{L^{p+2}(\R^n)}^{p+2}\Big)<\infty.
\end{equation}
\end{proposition}
The estimate \eqref{YT} plays a fundamental role in \cite{yt} and will be of crucial importance in the sequel
along the proof of Theorem \ref{main}. The basic idea to establish \eqref{YT} is to multiply first the equation \eqref{w} by $t^{\alpha(n,p)}$ and
then in a second step the corresponding equation is tested with the function $\partial_t \bar w_\varphi(t,x)$. Then the proof follows by integration by parts
and by considering the real part of the identity obtained. Concerning the property
\eqref{sigmaw} follows from the definition of $w_\varphi(t,x)$
and from the fact that $\varphi \in \Sigma_n$  implies $u_\varphi(t,x)\in C((0,\infty);\Sigma_n)$ (see \cite{c} for a proof of this fact).

\section{Proof of Theorem \ref{main}}

We shall treat in details the case $t\rightarrow \infty$ (by a similar argument one can treat $t\rightarrow - \infty$).
Recall that $e^{it\Delta}$ is a family of isometries in $H^1(\R^n)$ then \eqref{H1convergence}
is equivalent to show that
\begin{equation}\label{heart}
e^{-it\Delta} u(t,x)\overset{t\rightarrow \infty} \longrightarrow \varphi_+ \hbox{ in } H^1(\R^n).\end{equation}
Notice that by assumption $\varphi_+$ belongs to $\L^2(\R^n)$.
In the first elementary lemma below we show that indeed
$\varphi_+\in H^1(\R^n)$ and moreover
\begin{equation}\label{heartweak}
e^{-it\Delta} u(t,x)\overset{t\rightarrow \infty} \rightharpoonup  \varphi_+ \hbox{ in } H^1(\R^n).\end{equation}
We claim that \eqref{heart} follows provided that we show
\begin{equation}\label{convnormgrad}
\|\nabla (e^{-it\Delta} u_\varphi(t,x))\|_{L^2(\R^n)}\overset{t\rightarrow \infty}\longrightarrow  \|\nabla \varphi_+\|_{L^2(\R^n)}.\end{equation}
In fact by combining \eqref{convnormgrad} with the following convergence
$$\|e^{-it\Delta} u_\varphi(t,x)\|_{L^2(\R^n)}\overset{t\rightarrow \infty}\longrightarrow  \|\varphi_+\|_{L^2(\R^n)}$$
(which in turn follows from \eqref{assL2}) we get
\begin{equation}\label{convnorm}
\|e^{-it\Delta} u_\varphi(t,x)\|_{H^1(\R^n)}\overset{t\rightarrow \infty}\longrightarrow  \|\varphi_+\|_{H^1(\R^n)}.\end{equation}
Then we have weak convergence in $H^1(\R^n)$ by \eqref{heartweak} and convergence of the norms by \eqref{convnorm},
which together imply strong convergence in $H^1(\R^n)$, namely \eqref{heart}.
Hence the main difficulty it to establish \eqref{convnormgrad} which will follow from two separate lemmas, Lemma \ref{nablatilde} and Lemma \ref{exterior},
plus an extra argument that we partially borrow from \cite{tv}, where the precise long time behavior of moments is considered for a family of mass supercritical NLS.
\begin{lemme}\label{weakconv}
Under the assumptions of Theorem \ref{main} we have $\varphi_+\in H^1(\R^n)$ and $e^{-it \Delta} (u_\varphi(t,x)) \overset{t\rightarrow \infty} \rightharpoonup \varphi_+$ in $H^1(\R^n)$.
\end{lemme}
\begin{proof}
By conservation of the mass and energy (see \eqref{energy}) we have $\sup_t \|u_\varphi(t, x)\|_{H^1(\R^n)}<\infty$ and hence
\begin{equation}\label{unifbound}\sup_t \|e^{-it \Delta} u_\varphi(t,x)\|_{H^1(\R^n)}<\infty.\end{equation}
On the other hand the  assumption \eqref{assL2} implies
$$\|e^{-it \Delta} u_\varphi(t,x)-\varphi_+\|_{L^2(\R^n)}=\|u_\varphi(t,x)- e^{it\Delta} \varphi_+\|_{L^2(\R^n)}\overset{t\rightarrow \infty} \longrightarrow 0,$$
namely we have convergence of $e^{-it \Delta} u_\varphi(t,x)$ to $\varphi_+$ in $L^2(\R^n)$. By combining this fact with \eqref{unifbound}
we conclude $\varphi_+\in H^1(\R^n)$, as well as the weak convergence of $e^{-it \Delta} (u_\varphi(t,x))$ to $\varphi_+$
in $H^1(\R^n)$.
\end{proof}

\begin{lemme}\label{nablatilde}
Under the assumptions of Theorem \ref{main} we have the following property:
$$\Big \|\nabla u_\varphi(t,x)-i \frac {x}{2t} u_\varphi(t,x)\Big \|_{L^2(\R^n)}\overset{t\rightarrow \infty} \longrightarrow 0.$$
\end{lemme}

\begin{proof}
By using \eqref{defpseudo} we get
$$\nabla w_\varphi(t,x)=\frac{1}{t^{n/2+1}} \nabla \bar u_\varphi(\frac 1t, \frac xt) e^{i\frac{|x|^2}{4t}}+ \frac{i }{2 t^{n/2+1}} x \bar u_\varphi(\frac 1t, \frac xt) e^{i\frac{|x|^2}{4t}}$$
and hence
$$\|\nabla w_\varphi(t,x)\|_{L^2(\R^n)}=   \Big \| \frac 1t \nabla \bar u_\varphi(\frac 1t, x) + i  \frac{x}{2} \bar u_\varphi(\frac 1t, x)\Big \|_{L^2(\R^n)}.$$
Then we get
$$\Big \|\nabla u_\varphi(\frac 1t, x) - i \frac {tx} 2 u_\varphi(\frac 1t, x)\Big \|_{L^2(\R^n)}=t\|\nabla w_\varphi(t,x)\|_{L^2(\R^n)}, \quad \forall t\in (0,1]$$
and by \eqref{YT} we have
$$\Big \|\nabla u_\varphi(\frac 1t, x) - i \frac {tx} 2 u_\varphi(\frac 1t, x)\Big \|_{L^2(\R^n)}=O(t^{-\frac{\alpha(n,p)} 2+1}).$$
We conclude by considering the limit as $t\rightarrow 0^+$ (and hence $\frac 1t\rightarrow \infty$) and by noticing the $-\frac{\alpha(n,p)} 2+1>0$.
\end{proof}

\begin{lemme}\label{exterior} Under the same assumptions as in Theorem \ref{main} we have the following:
\begin{equation}\label{vircon}
\forall \varepsilon>0 \quad \exists t_\varepsilon, R_\varepsilon>0
\hbox{ s.t. } \sup_{t>t_\varepsilon} \int_{|x|>R_\varepsilon t} \frac{|x|^2}{t^2} |u_\varphi(t,x)|^2 dx<\varepsilon.\end{equation}
\end{lemme}

\begin{proof}
We have the identity
$$|w_\varphi(s,x)|^2=\frac {1}{s^{n}} |u_\varphi(\frac 1s, \frac xs)|^2$$
and hence
$$\int_{|x|>R} |x|^2 |w_\varphi(s,x)|^2 = \int_{|x|>R} |x|^2  |u_\varphi(\frac 1s, \frac xs)|^2 \frac{dx}{s^n}=
s^2 \int_{s|x|>R} |x|^2 |u_\varphi(\frac 1s, x)|^2dx.$$
If we denote $s=\frac 1t$ we get the following identity:
$$ \int_{|x|>R t} \frac{|x|^2}{t^2} |u_\varphi(t,x)|^2 dx=\int_{|x|>R} |x|^2|w_\varphi(\frac 1t,x)|^2 dx, \quad \forall R>0$$
hence in order to get the conclusion \eqref{vircon} we are reduced to prove:
\begin{equation}\label{cyl}\forall \varepsilon>0 \quad \exists \tilde t_\varepsilon, \tilde R_\varepsilon
\hbox{ s.t. } \sup_{t\in (0, \tilde t_\varepsilon]} \int_{|x|>R_\varepsilon} |x|^2 |w_\varphi(t, x)|^2 dx<\varepsilon.\end{equation}
More precisely showing smallness of the contribution to the renormalized second order moment of $u_\varphi(t,x)$ for large times in the exterior of a cone is equivalent to showing
smallness of the contribution to the second order moment of $w_\varphi(t,x)$ for small times in the exterior of a cylinder.
In order to prove \eqref{cyl} first we introduce a non-negative function $\varphi\in C^\infty(\R)$ such that:
\begin{enumerate}
\item $\varphi(x)=\varphi(|x|)$;
\item $\varphi(x)=|x|, \quad \forall |x|>1$;
\item $\varphi(x)=0, \quad \forall |x|<\frac 12$.
\end{enumerate}
Along with $\varphi$ we introduce the rescaled functions $\varphi_R(x)=R\varphi\big (\frac xR \big )$.
Then by elementary computations we get:
\begin{multline*}
\Big |\frac d{dt} \int_{\R^n} (\varphi_R(x))^2 |w_\varphi(t,x)|^2 dx \Big |\leq C \int_{\R^n} \varphi_R(x) |\nabla \varphi_R(x)|  |w_\varphi(t,x)| |\nabla w_\varphi(t,x)| dx
\\\leq C \|\nabla w_\varphi(t,x)\|_{L^2(\R^n)} \big(\int_{\R^n} (\varphi_R(x))^2 |w_\varphi(t,x)|^2 dx\big)^\frac 12\end{multline*}
and hence if we denote $w_{\varphi,R}(t)= \int_{\R^n} (\varphi_R(x))^2 |w_\varphi(t,x)|^2 dx$
we get
\begin{equation}\label{ode}\Big |\frac d{dt} w_{\varphi,R}(t)\Big |\leq C t^{-\frac \alpha 2} \sqrt {w_{\varphi,R}(t)}
\end{equation}
where we have used \eqref{YT} to estimate $\|\nabla w_\varphi(t,x)\|_{L^2(\R^n)}$.
Notice that in order to conclude \eqref{cyl}  it is sufficient to show that
for every $\varepsilon>0$ there exist $\tilde t_\varepsilon, \tilde R_\varepsilon>0$ such that
\begin{equation}\label{equiveq}\sup_{t\in (0, \tilde t_\varepsilon]}w_{\varphi, \tilde R_\varepsilon}(t)<\varepsilon.\end{equation}
In order to select $\tilde t_\varepsilon, \tilde R_\varepsilon>0$ with this property notice that by \eqref{ode} for every given $\bar s, \bar t \in (0,1]$ with $\bar s<\bar t$ we get:
\begin{multline}\label{gronw}
\sup_{t\in (\bar s , \bar t]} {w_{\varphi, R}}(t)\leq {w_{\varphi,R}}(\bar t) + C\sqrt{\sup_{t\in (\bar s, \bar t]} {w_{\varphi,R}}(t)} \int_{\bar s}^{\bar t}  \tau^{-\frac \alpha 2} d\tau
\\
\leq  {w_{\varphi,R}}(\bar t) + \frac 12 \sup_{t\in (\bar s, \bar t]} {w_{\varphi,R}}(t) + \frac {C^2}{2} \Big (\int_0^{\bar t}  \tau^{-\frac \alpha 2} d\tau\Big )^2, \quad \forall R>0.
\end{multline}
Next we select first $\tilde t_\varepsilon>0$ such that
$$\frac {C^2}2 \Big (\int_0^{\tilde t_\varepsilon}  \tau^{-\frac \alpha 2} d\tau\Big )^2 < \frac \varepsilon 4,$$ and in a second step we select $R_\varepsilon>0$ such that
$${w_{\varphi, R_\varepsilon}}(\tilde t_\varepsilon)<\frac \varepsilon 4,$$
Notice that this choice of
$\tilde t_\varepsilon$ is possible due to the integrablity of $\tau^{-\frac{\alpha}{2}}$ in $(0,1)$ and the choice of $\tilde R_\varepsilon>0$ is possible by the property
$w_\varphi (\tilde t_\varepsilon,x)\in \Sigma_n$ (see \eqref{sigmaw}).
Then
we conclude by \eqref{gronw}, where we choose $\bar t =\tilde t_\varepsilon$, $R=\tilde R_\varepsilon$ and $\bar s\in (0, \tilde t_\varepsilon)$,
the following estimate
$$\sup_{t\in (\bar s , \tilde t_\varepsilon]} w_{\varphi, \tilde R_\varepsilon}(t)<\varepsilon, \quad \forall \bar s\in (0, \tilde t_\varepsilon).$$
By the arbitrarity of $\bar s$ we conclude the estimate \eqref{equiveq} by taking $\bar s\rightarrow 0^+$. We point out that we have worked first for $\bar s>0$ and at the end we passed to the limit as
$\bar s\rightarrow 0^+$ since $\sup_{t\in (\bar s , 1]}  w_{\varphi, R}(t)<\infty$ for every $R>0$ (see \eqref{sigmaw}),
on the contrary in principle \eqref{sigmaw} does not imply $\sup_{t\in (0 , 1]}  w_{\varphi, R}(t)<\infty$
and hence \eqref{gronw} for $\bar s=0$ could become trivial since we could have two infinity quantities on the two sides.

\end{proof}

\begin{proof}[Proof of Theorem \ref{main}]
As already mentioned at the beginning of the section the key point is to establish \eqref{convnormgrad}. First notice that
the operator $\nabla$ and the group $e^{it\Delta}$ commute, and moreover $e^{it\Delta}$ are isometries in $L^2(\R^n)$.
Hence \eqref{convnormgrad} is equivalent to show
$$\|\nabla u_\varphi(t,x)\|_{L^2(\R^n)}\overset{t\rightarrow \infty} \longrightarrow \|\nabla \varphi_+\|_{L^2(\R^n)},$$
which in turn by Lemma \ref{nablatilde} is equivalent to
\begin{equation}\label{fulldomain}\Big \|\frac {x}{2t} u_\varphi(t,x)\Big \|_{L^2(\R^n)}\overset{t\rightarrow \infty} \longrightarrow \|\nabla \varphi_+\|_{L^2(\R^n)}.\end{equation}
Next we show that \eqref{fulldomain} is almost satisfied if we compute the $L^2$ norm in the more restricted region inside the cone $|x|<Rt$, for $R>0$
that will be choosen larger and larger.
More precisely we shall prove the following fact:
\begin{equation}\label{tulo}\int_{|x|<Rt} \frac{|x|^2}{t^2} |u_\varphi(t,x)|^2 dx\overset{t\rightarrow \infty} \longrightarrow 4 \int_{|x|<\frac R2}  |x|^2 |\hat \varphi_+(x)|^2 dx.
\end{equation}
By combining this property with \eqref{vircon} and by noticing that
$$\int_{|x|<\frac R2}  |x|^2 |\hat \varphi_+(x)|^2 dx \overset{R\rightarrow \infty} \longrightarrow
\int_{\R^n}  |x|^2 |\hat \varphi_+(x)|^2 dx =\|\nabla \varphi_+\|_{L^2}^2$$
we conclude \eqref{fulldomain}.
 In order to prove \eqref{tulo} we shall use the following asymptotic formula to describe free waves (see \cite{d} and \cite{rs}):
\begin{equation}\label{saymptoticbehavior}
\Big \|e^{it\Delta} h - \frac {e^{i\frac{|x|^2}{4t}}}{(2i t)^\frac n2} \hat h(\frac {x}{2t})\Big \|_{L^2(\R^n)}\overset{t\rightarrow \infty}\longrightarrow 0,
\quad \forall\, h\in L^2({\R}^n),
\end{equation}
where $\hat{h}(\xi)$ denotes the Fourier transform of $h$.
For every $R>0$ fixed we get by the Minkowski inequality
\begin{multline}\label{import}\left \|\frac{|x|}{t} \big (u_\varphi(t,x) - \frac {e^{i\frac{|x|^2}{4t}}}{(2i t)^\frac n2} \hat \varphi_+(\frac {x}{2t})
\big)
\right \|_{L^2(|x|<R t)}\leq
\left \|\frac{|x|}{t}\big (u_\varphi (t,x) - e^{it\Delta} \varphi_+\big)
\right \|_{L^2(|x|<R t)}\\+
\left \| \frac{|x|}{t} \big (e^{it\Delta} \varphi_+ -
\frac {e^{i\frac{|x|^2}{4t}}}{(2i t)^\frac n2} \hat \varphi_+(\frac {x}{2t})
\big)
\right \|_{L^2(|x|<Rt)}.
\end{multline}
Next notice that for any fixed $R>0$  inside the cone $|x|<Rt$ we have that the weight
$\frac{|x|}t$ is uniformly bounded and hence by combining \eqref{saymptoticbehavior} with \eqref{assL2}
we conclude that both terms on the r.h.s. in \eqref{import} converge to zero as $t\rightarrow \infty$, and hence we conclude:
$$\left \|\frac{|x|}{t} \big (u_\varphi (t,x) - \frac {e^{i\frac{|x|^2}{4t}}}{(2i t)^\frac n2} \hat \varphi_+(\frac {x}{2t})
\big)
\right \|_{L^2(|x|<R t)} \overset{t\rightarrow \infty} \longrightarrow 0.$$
From this fact we deduce
\begin{equation*}
\int_{|x|<Rt} \frac{|x|^2}{t^2} |u_\varphi(t,x)|^2 dx- \int_{|x|<Rt}  \frac{|x|^2}{t^2}|\hat \varphi_+(\frac {x}{2t})|^2 \frac{dx}{(2t)^n}
\overset{t\rightarrow \infty} \longrightarrow 0,\end{equation*}
that by a change of variable implies \eqref{tulo}.

\end{proof}
\section{Proof of Theorem \ref{mainmoment}}

Since we have to prove an equivalence we show separately the two implications.\\

{\em Proof of $\Rightarrow $}. Recall  that we have the property $u_\varphi(t,x)\in \Sigma_n$ for every $t>0$, since $\varphi\in \Sigma_n$ (see \cite{c}).
Moreover by the assumption $$\Big \|\frac{|x|}t (u_\varphi(t,x)- e^{it\Delta} \varphi_+)\Big \|_{L^2(\R^n)}\overset{t\rightarrow \infty} \longrightarrow 0$$
and in particular we have that there exists $\bar t$ such that
\begin{equation}\label{finite}\||x|(u_\varphi (\bar t,x)- e^{i\bar t\Delta} \varphi_+)\|_{L^2(\R^n)}<\infty.\end{equation}
Since we know that $u_\varphi(\bar t,x)\in \Sigma_n$ necessarily we have by \eqref{finite} and the Minkowski inequality that $e^{i\bar t\Delta} \varphi_+\in \Sigma_n$ and hence,
by the invariance of the space $\Sigma_n$ under the linear flow $e^{it\Delta}$, we deduce
$$\varphi_+=e^{-i\bar t\Delta} (e^{i\bar t\Delta} \varphi_+) \in \Sigma_n.$$
\\
\\
{\em Proof of $\Leftarrow $}.
First notice that the conclusion of Lemma \ref{exterior} is true even if we replace the nonlinear wave $u_\varphi(t,x)$ by the linear wave $e^{it\Delta}\varphi_+$.
In fact the proof of this fact in the linear case is even easier since once we compute the equation solved by the pesudo-conformal transformation of the linear wave $e^{it\Delta}\varphi_+$ we get again the linear Schr\"odinger equation and hence the norm of the gradient of the transformed solution is constant on the interval $(0,1]$ (which is even better than the bound that we get  in the nonlinear case, see \eqref{YT}). Hence the conclusion of  Lemma \ref{exterior} follows in the linear setting, by miming  {\em mutatis mutandis} the
proof given in the nonlinear setting.
Summarizing we have
\begin{equation}\label{virconnew}
\forall \varepsilon>0 \quad \exists t_\varepsilon, R_\varepsilon>0
\hbox{ s.t. } \sup_{t>t_\varepsilon} \max \Big\{\int_{|x|>R_\varepsilon t} \frac{|x|^2}{t^2} |u_\varphi(t,x)|^2 dx,
\int_{|x|>R_\varepsilon t} \frac{|x|^2}{t^2} |e^{it\Delta} \varphi_+|^2 dx\Big \}<\varepsilon.
\end{equation}
On the other hand for every fixed $R>0$ we have by \eqref{assL2}
\begin{equation}\label{assL2new}\Big \|\frac{|x|}t (u_\varphi(t,x)- e^{it\Delta} \varphi_+)\Big \|_{L^2(|x|<Rt)}\overset{t\rightarrow \infty} \longrightarrow 0, \quad \forall R>0,
\end{equation}
where we used that the weight $\frac{|x|}t$ is bounded inside the cone $|x|<Rt$.
The conclusion follows by combining \eqref{virconnew} and \eqref{assL2new}.

\section{Proof of Theorem \ref{equivreg}}

We work with $\varphi^+$ and $w_\varphi^+$ (the same proof works for $\varphi^-$ and $w_\varphi^-$).
We claim that we have the following identity
\begin{equation}\label{immid}\hat \varphi_+(\xi)= (2i)^\frac n2 \bar w^+_\varphi ( 2 \xi ).\end{equation}
Once this formula is proved then the conlusion follows by elementary Fourier analysis.
In order to prove \eqref{immid} first recall that by definition we have
$$\Big \|\frac{1}{t^{n/2}} \bar u_\varphi(\frac 1t, \frac xt) e^{i\frac{|x|^2}{4t}}-w^+_\varphi(x)\Big\|_{L^2(\R^n)}\overset{t\rightarrow 0^+}\longrightarrow 0$$
which in turn is equivalent to
$$\Big \|\frac{1}{t^{n/2}} \bar u_\varphi(\frac 1t, \frac xt) - e^{-i\frac{|x|^2}{4t}}w^+_\varphi (x)\Big\|_{L^2(\R^n)}\overset{t\rightarrow 0^+}\longrightarrow 0$$
and by the change of variable formula $\frac x t=y$ can be written as follows
$$\Big \|\bar u_\varphi(\frac 1t, y) - t^\frac n2  e^{-i\frac{t|y|^2}{4}}w^+_\varphi (t y)\Big\|_{L^2(\R^n)}\overset{t\rightarrow 0^+}\longrightarrow 0.$$
By introducing $s=\frac 1t$ we get
$$\Big \|\bar u_\varphi(s, y) - \frac 1{s^\frac n2}  e^{-i\frac{|y|^2}{4s}}w^+_\varphi (\frac y s)\Big\|_{L^2(\R^n)}\overset{s\rightarrow \infty}\longrightarrow 0$$
and by taking conjugate
\begin{equation}\label{w=}
\Big \|u_\varphi(s, y) - \frac 1{s^\frac n2}  e^{i\frac{|y|^2}{4s}}\bar w^+_\varphi (\frac y s)\Big\|_{L^2(\R^n)}\overset{s\rightarrow \infty}\longrightarrow 0.\end{equation}
Notice that by \eqref{saymptoticbehavior} we have
\begin{equation}\label{h+}\|e^{is\Delta} h_+ - \frac 1{s^\frac n2}  e^{i\frac{|y|^2}{4s}}\bar w^+_\varphi (\frac y {s})\Big\|_{L^2(\R^n)}\overset{s\rightarrow \infty}\longrightarrow 0\end{equation}
where
\begin{equation}\label{ident}(2i)^\frac n2 \bar w^+_\varphi (2y )=\hat h_+(y),\end{equation}
and hence by \eqref{w=} and \eqref{h+} we have
$$\|u_\varphi(s, y)-e^{is\Delta} h_+\|_{L^2(\R^n)}\overset{s\rightarrow \infty} \longrightarrow 0.$$
On the other hand we are assuming \eqref{assL2} and by uniqueness of scattering state we get $\varphi_+=h_+$
and hence by \eqref{ident} we get \eqref{immid}.

\section{Appendix: scattering in $\Sigma_n$ via lens transform, $p_n\leq p<\frac 4n$}

The aim of this appendix is to provide alternative proof of the results established in \cite{cw} and \cite{ts} by using the lens transform (instead of the pseudoconformal energy
which is the key tool in \cite{cw} and \cite{ts}).
We introduce, following \cite{bt}, \cite{btt}, \cite{carles}, for every time $t\in (-\frac \pi 4, \frac \pi 4)$ the lens transform acting as follows on time independent function $G:\R^n\rightarrow \R$:
$${\mathcal L}_t G (x)=(\cos (2t))^{-\frac n2}G\Big (\frac x{\cos(2t)}\Big ) e^{-i\frac{|x|^2 \tan (2t)}2}, \quad x\in \R^n.$$
By direct computation we have that if we denote
\begin{equation}\label{harmos}H=-\Delta + |x|^2,\end{equation}
then we get the following identity:
\begin{equation}\label{lin}e^{i(t(s))H} = {\mathcal L}_{t(s)} \circ e^{is\Delta}, \hbox{ where }  t(s)=\frac{\arctan (2s)}{2}.\end{equation}
Moreover we have that if the function $u_\varphi(t,x)$ is solution to \eqref{u} then  the function $v_\varphi(t,x)$ defined as follows:
\begin{equation}\label{nonlin}v_\varphi(t(s), x ):=  {\mathcal L}_{t(s)} \Big(u_\varphi(s,\cdot)\Big)(x)\end{equation}
solves the following Cauchy problem
\begin{equation}\label{lenstr}
\begin{cases}i\partial_t v_\varphi -H v_\varphi + \cos(2t)^{-\alpha(n,p)} v_\varphi |v_\varphi|^p=0, \quad (t,x)\in (-\frac \pi 4, \frac \pi 4)\times \R^n, \quad \alpha(n,p)=2-\frac{np}2
\\
v_\varphi(0,x)= \varphi\in \Sigma_n,
\end{cases}
\end{equation}
where $H$ is defined in \eqref{harmos}. Notice that the main advantage of the lens transform compared with the pseudoconformal transform
is that the full norm $\Sigma_n$ is involved in the energy associated with \eqref{lenstr}, and not only the $H^1(\R^n)$
norm. Therefore the lens transform seems to be a suitable tool to study the scattering in $\Sigma_n$.\\

We recall that the Cauchy problem \eqref{lenstr} admits one unique solution
\begin{equation}\label{localreg}v_\varphi(t,x)\in C([0, \frac \pi 4); \Sigma_n)\cap L^r_{loc}([0, \frac \pi 4); {\mathcal W}^{1,s}(\R^n))\end{equation}
where $(r,s)$ is an admissible Strichartz couple (namely $\frac 2r+\frac ns=\frac n2$ and  $r\geq 2$ for $n\geq 3$, $r>2$ for $n=2$, $r\geq 4$ for $n=1$)
and ${\mathcal W}^{1,s}(\R^n))$ denotes the harmonic Sobolev spaces associated, namely
$${\mathcal W}^{1,s}(\R^n)=\{w\in L^s(\R^n) \hbox{ s.t. } H^\frac 12 w\in L^s(\R^n)\}$$
endowed with the norm $\|w\|_{{\mathcal W}^{1,s}(\R^n))}
=\|w\|_{L^s}+\|H^\frac s2 w\|_{L^s}$.
Following \cite{DG} one can show that for $1<s<\infty$ there exist $C>0$ such that
\begin{equation}\label{equivalence}
\frac 1{C} (\|\nabla u\|_{L^s(\R^n)} + \|\langle x\rangle u\|_{L^s(\R^n)})  \leq
\|\varphi\|_{{\mathcal W}^{1,s}(\R^n)}\leq  C (\|\nabla u\|_{L^s(\R^n)} + \|\langle x \rangle u\|_{L^s(\R^n)})  .
\end{equation}
Moreover it is well-known that Strichartz estimates are  available (locally in time) for the
group $e^{-itH}$, under the same numerology for which they are satisfied (globally in time) for $e^{-it\Delta}$ (they can be obtained simply by applying the lens transform). Hence we have all the tools
necessary to construct local solutions to \eqref{lenstr} by repeating {\em mutatis mutandis} the same computations
necessary to construct local solutions for the usual NLS. Notice that the chain-rule in the framework of the harmonic Sobolev spaces
are essentially reduced to the classical chain rule in the usual Sobolev spaces by \eqref{equivalence}.
In order to show that the solution can be extended on the full interval $[0, \frac \pi 4]$ with regularity \eqref{localreg} we can rely
on the following conservation law:
\begin{multline}\label{enlens}\frac d{dt} \Big ( \cos(2t)^{\alpha(n,p)}\|v_\varphi(t,x)\|_{\Sigma_n}^2 +\frac 1{p+2} \|v_\varphi(t,x)\|_{L^{p+2}(\R^n)}^{p+2}\Big)\\
=\|v
_\varphi(t,x)\|_{\Sigma_n}^2 \frac d{dt} \cos(2t)^{\alpha(n,p)}<0,\quad \forall t\in [0, \frac \pi 4)\end{multline}
whose proof follows the same argument to get \eqref{YT} in the context of the pseudoconformal transformation.
Since the weight $\cos(2t)^{\alpha(n,p)}$ has no zero in the interval $[0, \frac \pi 4)$ we have a control
of the $\Sigma_n$ norm of the solution up to time $t=\frac \pi 4$ and hence we can globalize in $[0,\frac \pi 4]$.
A similar discussion hold in the interval $[-\pi/4,0]$\\

We have the following result that reduces the question of scattering in $\Sigma_n$ for  $u_\varphi(t,x)$ solution to \eqref{u} (see \eqref{cw}) 
to the  extendibility (by continuity)
of the function $v_\varphi(t,x)$ up to time $t=\frac{\pi}{4}$  in the space $\Sigma_n$.
\begin{proposition}\label{sigmaeq}
Let $\varphi\in \Sigma_n$, $0<p<\frac 4{n-2}$ for $n\geq 3$ and $0<p<\infty$ for $n=1,2$. Then we have the following equivalence:
\begin{equation*}
\exists \varphi_+\in \Sigma_n \hbox{ s.t. } \|e^{-is\Delta} (u_\varphi(s,y))-\varphi_+\|_{\Sigma_n}\overset{s\rightarrow \infty}\longrightarrow 0
\iff \exists v_+\in \Sigma_n \hbox{ s.t. } \|v_\varphi(t,x)-v_+\|_{\Sigma_n}\overset{t\rightarrow {\frac {\pi} 4}^-} \longrightarrow 0.
\end{equation*}
\end{proposition}

\begin{proof}
The identity \eqref{lin}  is equivalent to
$$e^{-is\Delta}= e^{-i(t(s))H}  \circ {\mathcal L}_{t(s)}$$
and hence
\begin{equation}\label{pomors}e^{-is\Delta}(u_\varphi(s,y))= e^{-i(t(s))H} \Big({\mathcal L}_{t(s)} (u_\varphi(s,\cdot)\Big)=e^{-i(t(s))H} (v_\varphi(t(s), y))
\end{equation}
where we has used at the last step \eqref{nonlin}. We conclude since $e^{-i(t(s))H}$ are isometries in $\Sigma_n$ and
$\lim_{s\rightarrow \infty} t(s)=\frac \pi 4$.

\end{proof}

\subsection{The case $p_n<p<\frac 4{n-2}$}
In this subsection we provide an alternative proof of the following result first established in \cite{ts} (see also \cite{c}).
\begin{theoreme}[\cite{ts}]
Assume $p_n<p<\frac 4{n-2}$ for $n\geq 3$ and $p_n<p<\infty$ for $n=1,2$ (here $p_n$ is defined in \eqref{pnlens}). Then for every $\varphi \in \Sigma_n$ there exists $\varphi_+\in \Sigma_n$ such that
$$\|e^{-it\Delta} (u_\varphi(t,x))-\varphi_+\|_{\Sigma_n}\overset{t\rightarrow \infty} \longrightarrow 0.$$
\end{theoreme}

\begin{proof}

By Proposition \ref{sigmaeq} we have to prove
$\|v_\varphi(t,x)-v_+\|_{\Sigma_n}\overset{t\rightarrow {\frac {\pi} 4}^-} \longrightarrow 0$, where $v_+\in \Sigma_n$.
In the rest of the proof we shall denote $v=v_\varphi$.
Next we denote by
$(r,p+2)$, the couple of exponents such that
\begin{equation}\label{admiss}\frac 2r + \frac n{p+2}=\frac n2\end{equation}
and we shall first prove
$v\in L^r((0, \frac \pi 4); {\mathcal W}^{1, p+2}(\R^n)).$
It is easy to check that the couple $(r,p+2)$ is Strichartz admissible in any dimension $n\geq 1$.
In view of \eqref{localreg} it is sufficient
to prove the existence of $t_0\in (0, \frac \pi 4)$ such that $v\in L^r((t_0, \frac \pi 4); {\mathcal W}^{1, p+2}(\R^n)),$
and in turn it is sufficient to show that 
$\sup_{\tau\in (t_0, \frac \pi 4)} \| v\|_{L^r((t_0, \tau); {\mathcal W}^{1, p+2}(\R^n))}<\infty$.
Notice that the main advantage of working with $\tau<\frac \pi 4$ is that in the following computation we deal with finite quantities.
By Strichartz estimates available for the propagator $e^{it H}$
we get:
\begin{equation}\label{stricht0new}\| v\|_{L^r((t_0, \tau); {\mathcal W}^{1, p+2}(\R^n))}
\leq C\|v(t_0)\|_{\Sigma_n} + C \| \cos(2t)^{-\alpha(n,p)} v|v|^p\|_{L^{r'}((t_0, \tau); {\mathcal W}^{1, (p+2)'}(\R^n))}\end{equation}
where $t_0$ is an arbitrary point in the interval $[0, \frac \pi 4)$ that we shall fix later,  $r', p'$ denote conjugate exponents
and $\tau$ is arbitrary in $(t_0, \frac \pi 4)$.
Notice that by the chain rule and H\"older inequality w.r.t. space and time we can continue the estimate as follows
\begin{equation}\label{stricht0}\dots
\leq C\|v(t_0)\|_{\Sigma_n} + C \| \cos(2t)^{-\alpha(n,p)}\|_{L^\frac r{r-2}(t_0, \frac \pi 4) } \|v\|_{L^{r}((t_0, \tau); {\mathcal W}^{1, (p+2)}(\R^n))}
\|v\|_{L^\infty((t_0, \tau);L^{p+2}(\R^n))}^p.\end{equation}
Due to \eqref{enlens} (which implies $\sup_{t\in (0, \frac \pi4)} \|v(t,x)\|_{L^{p+2}}^{p+2}<\infty$) we can absorb the second term on the r.h.s. in \eqref{stricht0} in the l.h.s. in \eqref{stricht0new} provided we have
$\cos(2t)^{-\alpha(n,p)}\in {L^\frac r{r-2}(0, \frac \pi 4) }$ and $t_0$ is close enough to $\frac \pi 4$. This integrability condition is equivalent to
$\frac{\alpha(n,p) r}{r-2}<1$ which in turn, thanks to \eqref{admiss}, is equivalent to $np^2+(n-2)p-4>0$ (recall that $\alpha(n,p)=2 -\frac{np}2$). We conclude since we recall $p_n$ is the larger root of
the algebraic equation $nx^2+(n-2) x-4=0$.\\
To deduce the existence of the limit $v_+$ by the Duhamel formulation and dual of Strichzrtz estimates
we have for any couple $0<\tau<\sigma<\frac \pi 4$:
\begin{multline*}\|v(\tau)-v(\sigma)\|_{\Sigma_n}=
\Big \| \int_{\tau}^{\sigma} e^{i(t-s)H} \cos(2s)^{-\alpha(n,p)} v(s)|v(s)|^pds\Big \|_{\Sigma_n}\\=
\Big \| \int_{\tau}^{\sigma} e^{-isH} \cos(2s)^{-\alpha(n,p)} v(s)|v(s)|^pds\Big \|_{\Sigma_n}
\leq C \Big \| \cos(2s)^{-\alpha(n,p)} v(s)|v(s)|^p ds\Big\|_{L^{r'}((\tau, \sigma)); {\mathcal W}^{1, (p+2)'}(\R^n))}.
\end{multline*}
Arguing as above and by using $v\in L^r([t_0, \frac \pi 4]; {\mathcal W}^{1, p+2}(\R^n))$ we can continue as follows:
\begin{equation*}
\dots \leq C  \| \cos(2t)^{-\alpha(n,p)}\|_{L^\frac r{r-2}(\tau, \sigma ) }
\|v\|_{L^{r}((\tau, \sigma); {\mathcal W}^{1, (p+2)}(\R^n))}\overset{\tau, \sigma \rightarrow \frac \pi 4^-}\longrightarrow 0.
\end{equation*}

\end{proof}

\subsection{The case $p_n\leq p<\frac 4{n-2}$, $n\geq 3$}

Next result includes the one in \cite{ts} with the extra bonus that it covers the limit case $p=p_n$.
We shall give the proof for $n\geq 3$, however the result is true also for $n=1,2$. We recall that compared with the original proof in \cite{cw} we deal
with the equation obtained after the lens transform, which is adapted to work in the $\Sigma_n$ space, rather than the pseudoconformal transfomation
that seems to perform better in the $H^1(\R^n)$ setting. Another point is that we give a proof of the key alternative
\eqref{altern1} or \eqref{altern2} below, based on a continuity argument. This is different of the proof given in \cite{cw} based on a fixed point.
We restrict below to the case $n\geq 3$ however, following \cite{cw} the proof can be adapted to the case $n=1$, the case $n=2$ has been treated in \cite{on}.
\begin{theoreme}[\cite{cw}]
Assume $n\geq 3$ and $p_n\leq p<\frac 4{n-2}$ ($p_n$ is defined in \eqref{pnlens}), then for every $\varphi \in \Sigma_n$ there exists $\varphi_+\in \Sigma_n$ such that
$$\|e^{-it\Delta}(u_\varphi(t,x))-\varphi_+\|_{\Sigma_n}\overset{t\rightarrow \infty} \longrightarrow 0.$$
\end{theoreme}

\begin{proof}
By Proposition \ref{sigmaeq} we are reduced to prove
$\|v_\varphi(t,x)-v_+\|_{\Sigma_n}\overset{t\rightarrow {\frac {\pi} 4}^-} \longrightarrow 0$, where $v_+\in \Sigma_n$.
In the rest of the proof we shall denote $v=v_\varphi$.

We claim that we have the following alternative for every $\frac 4{n+2}<p<\frac 4{n-2}$ (notice $\frac 4{n+2}<p_n$):\\

- either there exists $v_+\in \Sigma_n$ such that
\begin{equation}\label{altern1}\|v(t,x)-v_+\|_{\Sigma_n}\overset{t\rightarrow {\frac \pi 4}^-}\longrightarrow 0;\end{equation}\\

- or we have the lower bound \begin{equation}\label{altern2}\inf_{t\in [0, \frac \pi 4)}
\|v (t,x)\|_{\Sigma_n}^p \Big (\int_t^{\frac \pi4} |\cos(2\tau)|^{-\frac{4 \alpha(n,p)}{4-p(n-2)}} d\tau\Big )^\frac{4-p(n-2)}{4}>0.\end{equation}
We shall prove first how the alternative \eqref{altern1} or \eqref{altern2} implies the result. We need to exclude the scenario \eqref{altern2}
under the extra condition $p_n\leq p<\frac 4{n-2}$.
Indeed if by the absurd \eqref{altern2} is true then
we get by \eqref{enlens}
\begin{multline*}\frac d{dt} \Big ( \cos(2t)^{\alpha(n,p)}\|v(t,x)\|_{\Sigma_n}^2 +\frac 1{p+2} \|v(t,x)\|_{L^{p+2}(\R^n)}^{p+2}\Big )
\\\leq -2 \varepsilon_0^\frac 2p  \alpha(n,p)  \sin (2t) \cos(2t)^{\alpha(n,p)-1} \Big (\int_t^{\frac \pi4} |\cos(2\tau)|^{-\frac{4\alpha(n,p)}{4-p(n-2)}} d\tau\Big)^{\frac{-4+p(n-2)}{2p}},
\quad t\in (0, \frac \pi 4),
\end{multline*}
where $\varepsilon_0>0$ is the infimum in \eqref{altern2}. Notice that $\cos(2t)$ behaves as $(\frac \pi 4-t)$ when $t\rightarrow \frac \pi 4^-$
and hence we get by elementary computations
\begin{equation*}\frac d{dt} \Big ( \cos(2t)^{\alpha(n,p)}\|v(t,x)\|_{\Sigma_n}^2 +\frac 1{p+2} \|v(t,x)\|_{L^{p+2}(\R^n)}^{p+2}\Big )
\\\leq -c_0 (t-\frac \pi 4)^{-1-\frac{np^2+p(n-2)-4}{2p}},\quad t\in (0, \frac \pi 4)
\end{equation*}
for a suitable $c_0>0$. Notice that the function at the r.h.s. fails to be integrable on $(0,\frac \pi 4)$ as long as $p\geq p_n$ and hence by integration
on the interval $(0, \frac \pi 4)$ we easily get a contradiction.\\

Next we give a proof of the alternative \eqref{altern1} or \eqref{altern2} which is based on the following remark.
\begin{lemme}\label{fk}
Given two sequences  $a_k, b_k>0$ for $k\in \N$ and $p>0$,  define $f_k:\R^+\rightarrow \R$ as follows
$f_k(s)=s-a_k - b_ks^{1+p}.$
Assume that $a_kb_k^\frac 1p\overset{k\rightarrow \infty} \longrightarrow 0$ then there exists $\bar k$ such that
for every $k>\bar k$ there exist $0<c_k<d_k<\infty$ such that
$$\{s\in \R^+ \hbox{ s.t. } f_k(s)\leq 0\}=[0, c_k]\cup [d_k,\infty).$$
\end{lemme}

\begin{proof}
One can check that the function $f_k$ has one unique maximum at the point $\bar s_k>0$ given by the condition
$f_k'(\bar s_k)=0$, namely $\bar s_k^p=\frac 1{(p+1)b_k}$. Moreover $f_k$ is increasing for $s<\bar s_k$, decreasing for $s>\bar s_k$,
$\lim_{s\rightarrow \infty} f_k(s)=-\infty$ and $f_k(0)<0$.
We conclude  provided that we show that for $k$ large enough we have $f_k(\bar s_k)>0$.
By direct computation we get
$f_k(\bar s_k)= \frac 1{(p+1)^\frac 1p b_k^\frac 1p} - a_k - \frac 1{(p+1)^{1+\frac 1p}b_k^\frac 1p}$
and hence the condition $f_k(\bar s_k)>0$ is equivalent to
$\frac 1{(p+1)^\frac 1p} - \frac 1{(p+1)^{1+\frac 1p}}>a_k b_k^\frac 1p$
which is satisfied for $k$ large enough due to the assumption $a_kb_k^\frac 1p\overset{k\rightarrow \infty} \longrightarrow 0$.

\end{proof}

We can now complete the proof of the alternative \eqref{altern1} or \eqref{altern2} in the general setting $\frac 4{n+2}<p<\frac 4{n-2}$.
Since now on we shall use that under this condition on $p$ we have $\int_{0}^{\frac \pi4} |\cos(2\tau)|^{-\frac{4\alpha(n,p)}{4-p(n-2)}} d\tau
<\infty$ that will be used since now on without any further comment.
We shall prove that if \eqref{altern2} is false then \eqref{altern1} is satisfied.
If \eqref{altern2} is false then there exists a sequence $t_k\in (0, \frac \pi 4)$ and $\varepsilon_k>0$ such that
\begin{equation}\label{infinitesimal}t_k\overset{k\rightarrow \infty} \longrightarrow \frac \pi 4
 \quad \hbox{ and } \quad \|v(t_k,x)\|_{\Sigma_n}^p \Big (\int_{t_k}^{\frac \pi4} |\cos(2\tau)|^{-\frac{4\alpha(n,p)}{4-p(n-2)}} d\tau
 \Big )^\frac{4-p(n-2)}{4}=\varepsilon_k\overset{k\rightarrow \infty} \longrightarrow 0.\end{equation}
In the sequel we denote $v(t_k,x)=v_k$.
Next we choose the Strichartz admissible couple $(r,q)$ such that
$$1-\frac 2q=\frac {p(n-2)}{2n}$$
and by Strichartz estimates and H\"older inequalities (in space and time)
\begin{multline}\label{stricht0newend}\| v\|_{L^r((t_k, t); {\mathcal W}^{1, q}(\R^n))}
\leq C\|v_k\|_{\Sigma_n} + C \| \cos(2t)^{-\alpha(n,p)} v|v|^p\|_{L^{r'}((t_k, t); {\mathcal W}^{1, \frac{2nq}{2n+pq(n-2)}}(\R^n))}\\
\leq C \|v_k\|_{\Sigma_n}  + C \Big (\int_{t_k}^{\frac \pi4} |\cos(2\tau)|^{-\frac{4\alpha(n,p)}{4-p(n-2)}} d\tau\Big )^\frac{4-p(n-2)}{4} \|v\|_{L^\infty((t_k, t); L^\frac{2n}{n-2}(\R^n))}^p \| v\|_{L^r((t_k, t); {\mathcal W}^{1, q}(\R^n))} \end{multline}
and by the Sobolev embedding and elementary inequalities, we can continue the estimate as follows
\begin{multline*}
\dots \leq C \|v_k\|_{\Sigma_n}  + C \Big (\int_{t_k}^{\frac \pi4} |\cos(2\tau)|^{-\frac{4\alpha(n,p)}{4-p(n-2)}} d\tau\Big )^\frac{4-p(n-2)}{4} \|v\|_{L^\infty((t_k, t); \Sigma_n)}^p \| v\|_{L^r((t_k, t); {\mathcal W}^{1, q}(\R^n))}\\
\leq C \|v_k\|_{\Sigma_n} + C \Big (\int_{t_k}^{\frac \pi4} |\cos(2\tau)|^{-\frac{4\alpha(n,p)}{4-p(n-2)}} d\tau\Big )^\frac{4-p(n-2)}{4} \|v-v_k\|_{L^\infty((t_k, t); \Sigma_n)}^p \| v\|_{L^r((t_k, t); {\mathcal W}^{1, q}(\R^n))}\\+ C \Big (\int_{t_k}^{\frac \pi4} |\cos(2\tau)|^{-\frac{4\alpha(n,p)}{4-p(n-2)}} d\tau \Big)^\frac{4-p(n-2)}{4} \|v_k\|_{\Sigma_n}^p \| v\|_{L^r((t_k, t); {\mathcal W}^{1, q}(\R^n))}.
\end{multline*}
Due to \eqref{infinitesimal} we have that if we choose $k$ large enough then the last term on the r.h.s. can be estimated by
$\frac 12 \| v\|_{L^r((t_k, t); {\mathcal W}^{1, q}(\R^n))}$. Hence we can absorb it on the l.h.s. and we get
\begin{multline}\label{diff2}
\| v\|_{L^r((t_k, t); {\mathcal W}^{1, q}(\R^n))}
\leq C \|v_k\|_{\Sigma_n} \\+ C \Big (\int_{t_k}^{\frac \pi4} |\cos(2\tau)|^{-\frac{4p}{4-p(n-2)}} d\tau\Big )^\frac{4-p(n-2)}{4} \|v-v_k\|_{L^\infty((t_k, t); \Sigma_n)}^p \| v\|_{L^r((t_k, t); {\mathcal W}^{1, q}(\R^n))}.
\end{multline}
Again by Strichartz estimates and triangular inequality we get
\begin{multline*}\| v-v_k\|_{L^\infty((t_k, t); \Sigma_n)}\leq \|v_k\|_{\Sigma_n} + \| v\|_{L^\infty((t_k, t); \Sigma_n)}
\\\leq C\|v_k\|_{\Sigma_n} + C \| \cos(2t)^{-\alpha(n,p)} v|v|^p\|_{L^{r'}((t_k, t); {\mathcal W}^{1, \frac{2nq}{2n+pq(n-2)}}(\R^n))}\end{multline*}
and hence we can estimate the r.h.s. as above and we get for $k$ large enough
\begin{multline}\label{diff}
\| v-v_k\|_{L^\infty((t_k, t); \Sigma_n)} \leq C \|v_k\|_{\Sigma_n} \\+ C \Big (\int_{t_k}^{\frac \pi4} |\cos(2\tau)|^{-\frac{4\alpha(n,p)}{4-p(n-2)}} d\tau
\Big )^\frac{4-p(n-2)}{4} \|v-v_k\|_{L^\infty
((t_k, t); \Sigma_n)}^p \| v\|_{L^r((t_k, t); {\mathcal W}^{1, q}(\R^n))}+ \frac 12\| v\|_{L^r((t_k, t); {\mathcal W}^{1, q}(\R^n))}.
\end{multline}
Next we introduce the functions $X_k:(t_k,\frac \pi 4)\rightarrow \R^+$ defined as follows
$X_k(t)=\|v-v_k\|_{L^\infty((t_k, t); \Sigma_n)}+ \| v\|_{L^r((t_k, t); {\mathcal W}^{1, q}(\R^n))}$.
Notice that by \eqref{diff2} and \eqref{diff} we get
$$X_k(t)\leq C \|v_k\|_{\Sigma_n} + C \Big (\int_{t_k}^{\frac \pi 4} |\cos(2\tau)|^{-\frac{4\alpha(n,p)}{4-p(n-2)}} d\tau \Big)^\frac{4-p(n-2)}{4} (X_k(t))^{p+1}$$
and hence $X_k(t)$ belongs to the sublevel $\{f_k(s)\leq 0\}$
where $f_k(s)$ is as in Lemma \ref{fk},
with $a_k=C \|v_k\|_{\Sigma_n} $ and $b_k=C (\int_{t_k}^{\frac \pi 4} |\cos(2\tau)|^{-\frac{4\alpha(n,p)}{4-p(n-2)}} d\tau)^\frac{4-p(n-2)}{4} $.
Notice that $a_kb_k^\frac 1p \overset{t\rightarrow \infty} \longrightarrow 0$ by \eqref{infinitesimal} and hence if we choose $k=\bar k+1$ 
(following the notations of the Lemma \ref{fk})
we get, since $X_{\bar k+1}(t)$ are continuous functions and $X_{\bar k+1}(t_{\bar k+1})=0$, that $X_{\bar k+1}(t)$ leaves for every $t\in (t_{\bar k+1}, \frac \pi 4)$ in the corresponding bounded
connected component $[0, c_{\bar k+1}]$ provided by Lemma \ref{fk}.
Summarizing we get $v(t,x)\in L^r((0,\frac \pi 4);  {\mathcal W}^{1, q}(\R^n))\cap L^\infty ((0,\frac \pi 4);  \Sigma_n)$.
Going back to the Duhamel formulation, using Strichartz estimates and H\"older inequality in space and time (in the same spirit as in \eqref{stricht0newend})
we get for every $0<\tau<\sigma<\frac \pi 4$:
\begin{multline*}\|v(\tau)-v(\sigma)\|_{\Sigma_n}=
\Big \| \int_{\tau}^{\sigma} e^{i(t-s)H} \cos(2s)^{-\alpha(n,p)} v(s)|v(s)|^pds \Big\|_{\Sigma_n}\\=
\Big \| \int_{\tau}^{\sigma} e^{-isH} \cos(2s)^{-\alpha(n,p)} v(s)|v(s)|^pds\Big \|_{\Sigma_n}
\leq C \| \cos(2s)^{-\alpha(n,p)} v(s)|v(s)|^p\|_{L^{r'}((\tau, \sigma); {\mathcal W}^{1, \frac{2nq}{2n+pq(n-2)}}(\R^n))}
\\\leq C \Big (\int_{0}^{\frac \pi4} |\cos(2\tau)|^{-\frac{4\alpha(n,p)}{4-p(n-2)}} d\tau \Big)^\frac{4-p(n-2)}{4} \|v\|_{L^\infty((\tau, \sigma); L^\frac{2n}{n-2}(\R^n))}^p \| v\|_{L^r((\tau, \sigma); {\mathcal W}^{1, q}(\R^n))}
\overset{\tau, \sigma\rightarrow \frac \pi 4^-}\longrightarrow 0.
\end{multline*}
where at the last step we have used the property 
$v(t,x)\in L^r((0,\frac \pi 4);  {\mathcal W}^{1, q}(\R^n))\cap L^\infty ((0,\frac \pi 4);  \Sigma_n)$ established above and the Sobolev embedding
$\Sigma_n\subset L^\frac{2n}{n-2}$.

\end{proof}


\end{document}